\documentclass[12pt]{amsart}
 
\pdfoutput=1
\usepackage[margin=1in]{geometry}
\usepackage{amsmath,amsthm,amssymb,bbm,color,esint,esvect,float,graphicx,mathrsfs}
\usepackage[color=green!40]{todonotes}
\usepackage{color}

\usepackage{combelow}
\usepackage{geometry}
\usepackage{microtype}
\usepackage{lmodern} 
\usepackage{stmaryrd}
\usepackage[normalem]{ulem}
\newcommand{\st}{\;:\;}
\usepackage[abbrev]{amsrefs}

\DeclareMathOperator{\linspan}{span}
\DeclareMathOperator{\Div}{div}

\newtheorem{thm}{Theorem}[section]
\newtheorem{lemma}[thm]{Lemma}
\newtheorem{cor}[thm]{Corollary}

\newtheorem{rem}[thm]{Remark}
\newtheorem{defin}[thm]{Definition}

\newtheorem{openquestion}[thm]{Open Question}

\DeclareMathOperator{\BV}{BV}
 
\numberwithin{equation}{section}

\newcommand{\set}[2]{\{{#1}\mid{#2}\}}

\theoremstyle{definition}
\DeclareMathOperator{\esssup}{ess \, sup}

\begin{document}


\title[An atomic decomposition for functions of bounded variation]{An atomic decomposition for functions of bounded variation}

\author{Daniel Spector}
\address{Daniel Spector\hfill\break\indent 
Department of Mathematics\hfill\break\indent 
National Taiwan Normal University\hfill\break\indent 
No. 88, Section 4, Tingzhou Road, Wenshan District, Taipei City, Taiwan 116, R.O.C.\hfill\break\indent
and \hfill\break\indent
National Center for Theoretical Sciences\hfill\break\indent 
No. 1 Sec. 4 Roosevelt Rd., National Taiwan
University, Taipei, 106, Taiwan\hfill\break\indent
and\hfill\break\indent
Department of Mathematics\hfill\break\indent 
University of Pittsburgh\hfill\break\indent 
Pittsburgh, PA 15261 USA
}
\email{spectda@gapps.ntnu.edu.tw}

\author{Cody B. Stockdale}
\address{Cody B. Stockdale\hfill\break\indent 
 School of Mathematical Sciences and Statistics\hfill\break\indent 
 Clemson University\hfill\break\indent 
 Clemson, SC 29634, USA}
\email{cbstock@clemson.edu}

\author{Dmitriy Stolyarov}
\address{Dmitriy Stolyarov\hfill\break\indent 
Department of Mathematics and Computer Science\hfill\break\indent 
St. Petersburg State University\hfill\break\indent 
14th Line 29b, Vasilyevsky Island, St. Petersburg, Russia, 199178; \hfill\break\indent }
\email{d.m.stolyarov@spbu.ru}

\begin{abstract}
In this paper, we give a decomposition of the gradient measure  
$Du$ of an arbitrary function of bounded variation $u$ into a sum of atoms $\mu=D\chi_{F}$, where $F$ is a set of finite perimeter.  The atoms further satisfy the support, cancellation, normalization, and size conditions:  {F}or each $\mu$, there exists a {cube} $Q$ such that $\operatorname*{supp}\mu\subset Q$, 
 $\mu(Q)=0$, $|\mu|(Q)\leq 1$, and, denoting by $p_t$ the heat kernel in $\mathbb{R}^d$, 
\begin{align*}
\esssup_{x \in \mathbb{R}^d, t>0} |t^{1/2} p_t \ast \mu (x)|   \leq \frac{1}{l(Q)^{d-1}}.
\end{align*}
Our proof relies on a sampling of the coarea formula and a new boxing identity. 
We present several consequences of this result, including Sobolev inequalities, dimension estimates, and trace inequalities.
\end{abstract}
\maketitle


\section{Introduction}\label{IntroductionSection}
For $d\ge 2$ and an open set $\Omega \subset \mathbb{R}^d$ with compact Lipschitz boundary, denote by $\BV(\Omega)$ the space of integrable functions $u$ on $\Omega$ whose distributional derivatives $Du$ are finite Radon measures in $\Omega$. This space of functions of bounded variation has been extensively developed in connection with its role in the problem of Plateau \cite{Giusti}, and has since found many further applications, for instance, in image segmentation and fracture mechanics \cite{AFP} and in divergence form quasilinear PDE \cite{Anzellotti}. The usefulness of $\BV(\Omega)$ stems from its simple definition and robust structure, and its study has been well-codified in the literature; see \cites{AFP,EG,Ziemer}.

Two key properties of $u \in \BV(\Omega)$ are Gagliardo and Nirenberg's Sobolev-type inequality\footnote{When $\Omega=\mathbb{R}^d$, the term $\|u\|_{L^{1}(\mathbb{R}^d)}$ is not necessary.}: 
\begin{align}\label{GagliardoNirenberg}
\|u\|_{L^{d/(d-1)}(\Omega)} \lesssim |Du|(\Omega) + \|u\|_{L^{1}(\Omega)},
\end{align}
and the absolute continuity of the variation of the measure derivative $Du$ with respect to the $(d-1)$-dimensional Hausdorff measure:
\begin{align}\label{DimensionEstimates}
|Du| \ll \mathcal{H}^{d-1}.
\end{align}
While \eqref{GagliardoNirenberg} is relatively easy to prove\footnote{When one has the right idea, as there were roughly two decades between Sobolev's demonstration of the analogous result for $p>1$ and the first proof of \eqref{GagliardoNirenberg}.}, \eqref{DimensionEstimates} is less straightforward.  In particular,~\eqref{GagliardoNirenberg} can be argued from a combination of  
another fundamental result for $\BV(\Omega)$, the coarea formula 
\begin{align}\label{coarea}
|Du|(B) = \int_{-\infty}^\infty |D\chi_{\{u>t\}}|(B)\;dt
\end{align}
for all Borel sets $B \subset \Omega$ (see \cite[Theorem 3.40]{AFP}), 
the weak-type inequality 
$$
\sup_{t>0}t|\{x \in \mathbb{R}^d : |u(x)|>t\}|^{(d-1)/d}\lesssim |Du|(\mathbb{R}^d),
$$
and an extension argument which utilizes the assumption that $\Omega$ has a compact Lipschitz boundary \cite[Proposition 3.21 on p.~131]{AFP}.
On the other hand, \eqref{DimensionEstimates} follows from a combination of \eqref{coarea} and the more subtle identity
$|D\chi_{\{u>t\}}|(B) =\mathcal{H}^{d-1}(B\cap \partial^* \{u>t\})$
for almost every $t \in \mathbb{R}$, see \cite[(3.63) on p.~159]{AFP}.

In this paper, we prove an atomic decomposition of the space $\text{BV}(\Omega)$, which, following the theory developed in \cite{DS}, easily yields \eqref{GagliardoNirenberg} and \eqref{DimensionEstimates}. In fact, it will be useful to first establish such a result for the homogeneous counterpart on all of space,
\begin{align*}
\dot{\BV}(\mathbb{R}^d):= \left\{ u\in L^1_{loc}(\mathbb{R}^d) : Du \in M_b(\mathbb{R}^d;\mathbb{R}^d)\right\},
\end{align*}
and then deduce results for $\BV(\Omega)$ {using} an extension operator. Here, the notation $M_b(\mathbb{R}^d,\mathbb{R}^d)$ denotes the space of measures on $\mathbb{R}^d$ with values also in $\mathbb{R}^d$ equipped with the total variation norm.  To this end, we recall the definition of $\beta$-atoms introduced in \cite{DS}.
\begin{defin}
A signed measure $\mu$ of finite total variation is a $\beta$-atom if there exists a cube $Q\subset \mathbb{R}^d$ with sides parallel to the coordinate axes such that
\begin{enumerate}
\item \label{supp}$\operatorname*{supp} \mu \subset Q$\textup;
\item \label{ave} $\mu(Q) = 0$\textup;
\item $\esssup_{x \in \mathbb{R}^d, t>0} |t^{(d-\beta)/2} p_t\ast \mu (x)|  \leq \frac{1}{l(Q)^{\beta}}\label{linfty}$\textup;\quad and
\item \label{l1} $|\mu|(\mathbb{R}^d) \leq 1$.
\end{enumerate}
\end{defin}
\noindent
Here, $p_t$ denotes the heat kernel and
\begin{align*}
p_t\ast \mu(x)=\frac{1}{(4\pi t)^{d/2}} \int_{\mathbb{R}^d} e^{-|x-y|^2/4t}\;d\mu(y).
\end{align*}

Our main result is the following atomic decomposition of the space $\dot{\BV}(\mathbb{R}^d)$.
\begin{thm}\label{AtomicDecomposition}
Let $u \in \dot{\BV}(\mathbb{R}^d)$.  There exist sets of finite perimeter $\{E_{i,n}\}_{n \in \mathbb{N}, i=1,\ldots,n}$, dyadic cubes $\{Q_{i,n}\}_{n \in \mathbb{N}, i=1,\ldots,n}$ with $Q_{i_1,n} \cap Q_{i_2,n} = \emptyset$ for each $n \in \mathbb{N}$ and $i_1\neq i_2$, and scalars $\{\lambda_{i,n}\}_{n \in \mathbb{N}, i=1,\ldots,n}$ such that for each $n \in \mathbb{N}, i=1,\ldots,n,$ and $l=1,\ldots,d$  the measure
\begin{align*}
\mu^l_{i,n} = [D\chi_{E_{i,n} \cap Q_{i,n}}]_l
\end{align*}
is a $(d-1)$-atom,
\begin{align*}
[Du]_l = \lim_{n \to \infty} \sum_{i=1}^n \lambda_{i,n} \mu^l_{i,n}\
\end{align*}
in the weak--star topology of the space of measures, and
\begin{align*}
\limsup_{n\to \infty}  \sum_{i=1}^n |\lambda_{i,n}|  \lesssim |Du|(\mathbb{R}^d).
\end{align*}
\end{thm}
\noindent
Here, $[Du]_l$ denotes the $l^{\text{th}}$ component of the vector-valued measure $Du$ and $A \lesssim B$ means that there is a constant $C>0$, independent of $A$ and $B$, such that $A\leq CB$.

By \cite[Proposition 3.21 on p.~131]{AFP}, this implies
\begin{cor}\label{AtomicDecompositionBV}
Let $\Omega \subset \mathbb{R}^d$ be an open set with compact Lipschitz boundary and suppose $u \in \BV(\Omega)$.  
There exist sets of finite perimeter $\{E_{i,n}\}_{n \in \mathbb{N}, i=1,\ldots,n}$, dyadic cubes $\{Q_{i,n}\}_{n \in \mathbb{N}, i=1,\ldots,n}$ with $Q_{i_1,n} \cap Q_{i_2,n} = \emptyset$ for each $n \in \mathbb{N}$ and $i_1\neq i_2$, and scalars $\{\lambda_{i,n}\}_{n \in \mathbb{N}, i=1,\ldots,n}$ such that for each $n \in \mathbb{N}, i=1,\ldots,n,$ and $l=1,\ldots,d$  the measure
\begin{align*}
\mu^l_{i,n} = [D\chi_{E_{i,n} \cap Q_{i,n}}]_l
\end{align*}
is a $(d-1)$-atom,
\begin{align*}
[D\overline{u}]_l = \lim_{n \to \infty} \sum_{i=1}^n \lambda_{i,n} \mu^l_{i,n}\
\end{align*}
in the weak--star topology of the space of measures, and
\begin{align*}
\limsup_{n\to \infty}  \sum_{i=1}^n |\lambda_{i,n}|  \lesssim |Du|(\Omega) + \|u\|_{L^{1}(\Omega)},
\end{align*}
where $\overline{u} \in  \BV(\mathbb{R}^d)$ satisfies $\overline{u}=u$ in $\Omega$.
\end{cor}

By \cite[Theorem A]{DS}, Theorem \ref{AtomicDecomposition} implies the Gagliardo-Nirenberg estimate~\eqref{GagliardoNirenberg}.
\begin{cor}\label{GN}
 The inequality
\begin{align*}
\|u\|_{L^{d/(d-1)}(\Omega)}  \lesssim |Du|(\Omega) + \|u\|_{L^{1}(\Omega)}
\end{align*}
holds for all $u \in \BV(\Omega)$.
\end{cor}

\setlength\marginparwidth{50pt}

To be more precise, \cite[Theorem A]{DS} postulates the inequality
\begin{align}\label{besov_lorentz}
    \|I_\alpha \mu\|_{\dot{B}_{d/(d-\alpha),1}^{0,1}(\mathbb{R}^d)}\lesssim \|\mu\|_{DS_\beta(\mathbb{R}^d)},
\end{align}
where the space on the left is the Besov--Lorentz space that embeds continuously into~$L^{d/(d-\alpha)}(\mathbb{R}^d)$ (for more details see the explanation after Theorem~$2$ in~\cite{Stolyarov2022}), 
\begin{align*}
DS_\beta&(\mathbb{R}^d):=\\ 
&\left\{\mu \in M_b(\mathbb{R}^d)\colon \mu = \lim_{n\to \infty}\sum_{i=1}^{n} \lambda_{i,n} \mu_{i,n},    \text{ $\mu_{i,n}$ $\beta$-atoms, } \limsup_{n \to \infty} \sum_{i=1}^n |\lambda_{i,n}|<+\infty \right\},
\end{align*}
and
\begin{align*}
I_\alpha \mu(x) = \frac{1}{\gamma(\alpha)} \int_{\mathbb{R}^d} \frac{d\mu(y)}{|x-y|^{d-\alpha}}
\end{align*}
is the Riesz potential of order $\alpha \in (0,d)$, where $\gamma(\alpha)$ is a suitable normalization constant.  
Corollary \ref{GN} then follows from \eqref{besov_lorentz}, an extension argument, and the following result, which is an immediate consequence of Theorem \ref{AtomicDecomposition}.
\begin{cor}\label{dimension_stable_embedding}
For each $l=1,\ldots,d$, the inequality
\begin{align*}
        \|[Du]_l\|_{\text{DS}_{d-1}(\mathbb{R}^d)} \lesssim |Du|(\mathbb{R}^d)
\end{align*}
holds for all $u \in \dot{\BV}(\mathbb{R}^d)$.
\end{cor}
\noindent

We next turn our attention to dimension estimates.  By \cite[Theorem G]{DS}, Theorem \ref{AtomicDecomposition} gives the following improved version of \eqref{DimensionEstimates} for the space $\dot{\BV}(\mathbb{R}^d)$. 
\begin{cor}\label{ac}
Let $Du \in M_b(\mathbb{R}^d)$. If $\epsilon>0$, then there exists $\delta>0$ such that $|Du|(B)<\epsilon$ whenever $B\subset \mathbb{R}^d$ is a Borel set with $\mathcal{H}^{d-1}(B)<\delta$.
\end{cor}

\begin{rem}
While the statement of \cite[Theorem G]{DS} asserts a dimension estimate for elements of $DS_\beta(\mathbb{R}^d)$, an examination of the proof shows that the result establishes the stronger absolute continuity property recorded in Corollary \ref{ac}.  
\end{rem}

By an extension argument, one deduces
\begin{cor}\label{ac_omega}
Let $u \in \BV(\Omega)$. If $\epsilon>0$, then there exists $\delta>0$ such that $|Du|(B)<\epsilon$ whenever $B\subset \Omega$ is a Borel set with $\mathcal{H}^{d-1}(B)<\delta$.
\end{cor}

Finally, we turn our attention to trace inequalities.  By \cite[Theorem B]{DS}, we have
\begin{cor}\label{BV_trace}
If $\alpha \in (1,d)$, then 
\begin{align*}
\|I_\alpha D\overline{u}\|_{L^{1}(\nu)} \lesssim |Du|(\Omega) + \|u\|_{L^{1}(\Omega)}
\end{align*}
for $u \in \BV(\Omega)$ and every Radon measure $\nu$ such that $\nu(B(x,r))\lesssim r^{d-\alpha}$, where $\overline{u}$ is an extension of $u$ in the sense of  \cite[Definition 3.20 on p.~130]{AFP}.
\end{cor}

The questions of Sobolev inequalities, dimension estimates, and trace inequalities are of interest for more general differential operators than~$\nabla$, for example, for the symmetric gradient or exterior derivative acting on $k$-forms for $k=1,\ldots d-1$.  
An early result in this direction is
the Korn--Sobolev inequality established by M.J. Strauss \cite{Strauss}*{p.\thinspace{}208} in 1971:  
\begin{align}\label{Strauss}
\|u \|_{L^{d/(d-1)}(\mathbb{R}^d)} \lesssim \| Eu \|_{L^1(\mathbb{R}^d;\mathbb{R}^d)}
\end{align}
for all vector fields $u \in C^\infty_c(\mathbb{R}^d,\mathbb{R}^d)$, where
\begin{align*}
    Eu:=   \tfrac{1}{2}\bigl(\nabla u + (\nabla u)^T)
\end{align*}
is the symmetric gradient.  This improves Gagliardo and Nirenberg's inequality for vector-valued functions, as
\begin{align*}
|Eu| \leq |\nabla u|
\end{align*}
with a failure of both the reverse inequality and the weaker integrated version
\begin{align*}
\| \nabla u \|_{L^1(\mathbb{R}^d;\mathbb{R}^d)} \not\lesssim \| Eu \|_{L^1(\mathbb{R}^d;\mathbb{R}^d)},
\end{align*}
the latter being an example of the celebrated non-inequality of Ornstein \cite{Ornstein1962}.

The inequality \eqref{Strauss} was a lone result in the world of vector-valued inequalities in $L^1$ that remained unexplored for several decades until interest was renewed through the establishment of further inequalities in this spirit by J. Bourgain and H. Brezis in their pioneering works \cite{BourgainBrezis2004, BourgainBrezis2007}.  In these papers, Bourgain and Brezis established results for Hodge systems with $L^1$ data, a principal example being the div-curl system in three variables
\begin{align*}
\begin{cases}\operatorname*{curl} Z &= F, 
\\
 \operatorname*{div} Z &= 0, 
 \end{cases}
\end{align*}
for which they established the surprising $L^1$ estimate
\begin{align}\label{Sobolev}
\|Z\|_{L^{3/2}(\mathbb{R}^3;\mathbb{R}^{3})} \lesssim \|F\|_{L^1(\mathbb{R}^3;\mathbb{R}^3)}.
\end{align}
The inequality \eqref{Sobolev} is the proverbial face that launched a thousand ships, as it led to questions of simpler proofs, extensions, applications, and refinements; see e.g.~\cite{BousquetVanSchaftingen2014, CVSYu2017, GRV2024, HernandezSpector2024, LanzaniStein2005, Mazya2007, Mazya2010, RSS, SpectorVanSchaftingen2019, Stolyarov2021, Stolyarov2022,  VanSchaftingen2004, VanSchaftingen2004one, VanSchaftingen2013} and also the expository articles~\cite{Spector2020, VanSchaftingen2014, VanSchaftingen2024}.

A natural question following the work of Bourgain and Brezis was to understand the extent to which such inequalities are valid.  More precisely, one wonders whether there is a characterization of the differential operators that support a Sobolev inequality in~$L^1$.  This question was answered by J. Van Schaftingen in \cite{VanSchaftingen2013} through the introduction of an algebraic condition: a homogeneous constant coefficient linear differential operator $ A \colon C^\infty_c(\mathbb{R}^d,\mathbb{R}^l) \to C^\infty_c(\mathbb{R}^d,\mathbb{R}^m)$ is said to be cancelling if
\begin{align}\label{cancelling}
 \bigcap_{\xi\in\mathbb{R}^d \setminus \{0\}}\  A(\xi)[\mathbb{R}^l]=\{0\},
\end{align}
where we identify $A$ with its symbol.  The significance of this cancelling condition can be understood by Van Schaftingen's result that for a $k^{\text{th}}$ order homogeneous constant coefficient linear differential operator ${A}$: the inequality
\begin{equation}\label{CocancellingInequality}
    \|\nabla^{k-1}f\|_{L^{d/(d-1)}(\mathbb{R}^d;\mathbb{R}^{N})}\lesssim \|Af\|_{L^1(\mathbb{R}^d;\mathbb{R}^m)}
\end{equation}
holds for all $f \in C^\infty_c(\mathbb{R}^d;\mathbb{R}^l)$ if and only if ${A}$ is elliptic and cancelling.  

At first appearance, the cancellation condition \eqref{cancelling} and its relation to the inequality \eqref{CocancellingInequality} is somewhat mysterious, though upon further examination one understands that the cancellation condition is equivalent to the assertion that $A$ does not have a fundamental solution, i.e. if
\begin{align*}
Af = e\delta_0
\end{align*}
for some $f \in \mathcal{S}'(\mathbb{R}^d;\mathbb{R}^l)$ and $e \in \mathbb{R}^m$, then $e=0$.  This ensures that the standard counterexample for the corresponding inequalities for Riesz potentials on $L^1$ are not admissible fields in the vector inequalities.  This is a first assertion towards a connection between measures {satisfying} an algebraic condition such as \eqref{cancelling} and the dimension of $A$, the smallest sets that measures of the form $Af$ can charge.  In particular, for a cancelling differential operator, one has
\begin{align*}
|Af|  \ll \mathcal{H}^1,
\end{align*}
where here and in the sequel $\mathcal{H}^\gamma$ denotes the Hausdorff measure of dimension $\gamma \in [0,d]$.  This leads one to ask 
\begin{openquestion}\label{dimension}
Suppose $A$ is a $k^{\text{th}}$ order homogeneous constant coefficient cancelling linear differential operator.  What is the largest value of $\beta=\beta(A) \in (0,d]$ such that
\begin{align*}
|Af|  \ll \mathcal{H}^\beta
\end{align*}
for every $f \in L^1_{loc}(\mathbb{R}^d;\mathbb{R}^l)$ such that $Af \in M_b(\mathbb{R}^d;\mathbb{R}^m)$? 
\end{openquestion}
The study of the connection between an algebraic condition on a vector-valued differential operator and the dimension of the associated space of measures emerged independently of the investigation of the validity of Sobolev inequalities, seemingly with its first systematic treatment in~\cite{RoginskayaWojciechowski2006} and later development in~\cite{APHF2019, Ayoush2023, AyoushWojciechowski2017, Dobronravov2024, Stolyarov2023, StolyarovWojciechowski2014}.  For general $A$, that $\beta \geq 1$ follows from the analysis in ~\cite{RoginskayaWojciechowski2006} and this result is sharp for $A=\operatorname*{div}: C^\infty_c(\mathbb{R}^d;\mathbb{R}^d) \to C^\infty_c(\mathbb{R}^d)$.  This suggests that it may be interesting to pose a refinement of the question where one imposes further algebraic structure on $A$ and asks for corresponding values of $\beta$ which are sharp with respect to the hypothesis.  This was the viewpoint taken in \cite{APHF2019}, where lower bounds with $\beta>1$ were obtained for operators whose $\ell$-wave cone is empty, while later it was shown in~\cite{Ayoush2023} that those bounds were not sharp {for some sophisticated~$A$. } Further consideration of these ideas prompts
\begin{openquestion}\label{dimension_question}
Suppose $A$ is a $k^{\text{th}}$ order homogeneous constant coefficient cancelling linear differential operator which is $\ell$-cancelling, i.e.~
\begin{align*}
    \bigcap_{\substack{W \subseteq \mathbb{R}^d\\ \dim W = \ell}}
      \linspan 
        \,
        \bigl\{ 
          A (\xi)[v] 
        \st 
          \xi \in W 
          \text{ and } 
          v \in \mathbb{R}^l
        \bigr\}
  =
    \{0\}
  .
\end{align*}
{Does it hold} that
\begin{align*}
|Af|  \ll \mathcal{H}^\ell
\end{align*}
for every $f \in L^1_{loc}(\mathbb{R}^d;\mathbb{R}^l)$ such that $Af \in M_b(\mathbb{R}^d;\mathbb{R}^m)$? 
\end{openquestion}
\noindent
The $\ell$-cancelling condition appeared in \cite{SpectorVanSchaftingen2019}, see also \cite{Raita_report}.

A resolution, or even partial resolution, of Open Question \ref{dimension_question} suggests one to consider whether one has trace inequalities analogous to \eqref{BV_trace} for a given operator $A$, where one requires $\alpha>\alpha_0(A)$.  One poses
\begin{openquestion}\label{trace_inequality_oq}
Suppose $A$ is a $k^{\text{th}}$ order homogeneous constant coefficient cancelling linear differential operator.  {Does it hold} that
\begin{align*}
\|I_\alpha Af\|_{L^{1}(\nu)} \lesssim \|Af\|_{L^1(\mathbb{R}^d;\mathbb{R}^m)}
\end{align*}
for all $Af \in L^1(\mathbb{R}^d;\mathbb{R}^m)$ and every Radon measure $\nu$ such that $\nu(B(x,r))\lesssim r^{d-\alpha}$ for $\alpha \in (\alpha_0(A),d]$ for some $\alpha_0(A) \in (d-\beta(A),d]$?
\end{openquestion}
\noindent
A similar open question for a dual class of fields which satisfy a co-cancelling constraint has been raised in work of two of the authors and {B.} Rai\c t\u a \cite[Open Question 1.1]{RSS}.  One expects that the endpoint estimate $\alpha=d-\beta$ may also require the introduction of a singular integral operator as in \cite[Conjecture 1.6]{RSS}. 

One observes that at the endpoint $\alpha_0(A) =d-\beta(A)$, the left-hand-side of the inequality in Open Question \ref{trace_inequality_oq} is
\begin{align*}
\int_{\mathbb{R}^d} |AI_{d-\beta} f(x)| d\nu(x),
\end{align*}
a quantity slightly smaller than the Riesz mutual energy
\begin{align*}
\mathcal{E}(|Af|,\nu):=\int_{\mathbb{R}^d} \int_{\mathbb{R}^d} \frac{d|Af|(y)d\nu(x)}{|x-y|^{\beta}},
\end{align*}
which, when evaluated on the diagonal~$\nu = |Af|$, is connected to the Fourier dimension of its argument; see Section~$3.5$ in~\cite{Mattila}.  As the Fourier dimension controls the Hausdorff dimension, this suggests that trace estimates may be useful in deriving partial results toward the (probably more difficult) critical dimension estimate.

A fundamental difficulty concerning these questions is that at present there are insufficiently many available tools for complete resolutions in this more general setting.  The consideration of these questions led the first and third author to introduce dimension stable spaces of measures in \cite{DS}.  These spaces $DS_\beta(\mathbb{R}^d)$ have the property that they admit all Besov--Lorentz scale Sobolev inequalities, their elements are absolutely continuous with respect to $\mathcal{H}^\beta$, and admit trace estimates for $\alpha>d-\beta$. 
 It is natural to unify {Open Questions }~\ref{dimension_question} and~\ref{trace_inequality_oq} into one using the spaces~$DS_\beta(\mathbb{R}^d)$.
 { \begin{openquestion}\label{space_question}
Suppose $A$ is a $k^{\text{th}}$ order homogeneous constant coefficient cancelling linear differential operator which is $\ell$-cancelling.  {Does it hold} that
\begin{align*}
\|[Af]_l\|_{DS_{\ell}(\mathbb{R}^d)}\lesssim \|Af\|_{L^1(\mathbb{R}^d;\mathbb{R}^m)}
\end{align*}
{for all $Af \in L^1(\mathbb{R}^d;\mathbb{R}^m)$ and each $l=1,\ldots,m$?}
\end{openquestion}}
\noindent
For divergence-free measures on $\mathbb{R}^d$, i.e. $A=d^*$, the co-exterior derivative acting on $2$-currents, an affirmative answer to this question was already known with the introduction of the spaces \cite[Example 1.10]{DS}.  However, a point left open in that work was the embedding of $\dot{\BV}(\mathbb{R}^d)$ into $DS_{d-1}(\mathbb{R}^d)$, which was the initial impetus for this paper.  In particular, as noted, our Corollary \ref{dimension_stable_embedding} is an immediate consequence of Theorem \ref{AtomicDecomposition}, though the conclusion is slightly weaker:  The atoms given in Theorem \ref{AtomicDecomposition} have more structure than {those of} $DS_{d-1}(\mathbb{R}^d)$, with the former given by weak derivatives of characteristic functions, related to the fact that Theorem \ref{AtomicDecomposition} is not only a decomposition, but in fact a characterization, of $\dot{\BV}(\mathbb{R}^d)$.

Theorem \ref{AtomicDecomposition} is argued by density, a sampling of the coarea formula, and our
\begin{lemma}\label{BoxingLemma}
    Let $\mathcal{D}$ be a system of dyadic cubes in $\mathbb{R}^d$. If $U\subseteq\mathbb{R}^d$ is open, bounded, and of finite perimeter, then there exist $\{Q_j\}_{j=1}^{\infty} \subseteq \mathcal{D}$ and $C>0$ such that 
    $$
    \chi_U = \sum_{j=1}^\infty \chi_{U\cap Q_j}
    $$
    and
    $$
        \ell(Q_j)^{n-1} \leq C|D\chi_U|(Q_j)
    $$
    for all $j \in \mathbb{N}$.
\end{lemma}

Lemma \ref{BoxingLemma} is a dyadic version of the so-called boxing inequality of W. Gustin \cite{G1960}.  In fact, it is perhaps more in the spirit of Gustin's work than many of the subsequent results on boxing inequalities \cite{AN2023,DGN1998,Federer,HL2013,KKST2008,PS2020} which utilize balls in their statement of the inequality, as the name boxing inequality refers to the use of cubes in his original paper.
Beyond its application toward the proof of Theorem \ref{AtomicDecomposition}, Lemma \ref{BoxingLemma} is interesting in its own right, as instead of covering an open set with other sets with desirable properties, it provides an identity between an open set of finite perimeter and a countable union of disjoint sets with such properties, which may be useful toward the deduction of further results for $\BV(\Omega)$/$\dot{\BV}(\mathbb{R}^d)$ functions.

The plan of the paper is as follows.  In Section \ref{lemma} we prove a lemma sampling the coarea formula which will be useful in our proof of Theorem \ref{AtomicDecomposition}.  In Section \ref{proofs} we prove Lemma \ref{BoxingLemma}, Theorem \ref{AtomicDecomposition}, and the corollaries.

\section{A Sampling Lemma}\label{lemma}

\begin{lemma}\label{sampling}
Let~$u\colon\mathbb{R}^d \to \mathbb{R}$ be a smooth compactly supported function. There exists a sequence of finite collections of open sets~$\{\Omega_{n,i}\}_{n\in \mathbb{N},i=1,\ldots,n}$ and scalars~$\{a_{n,i}\}_{n\in \mathbb{N},i=1,\ldots,n}$ such that
 one has
\begin{align}\label{weak}
\int_{\mathbb{R}^d} \Phi \cdot \nabla u \;dx = \lim_{n\to \infty} \sum_{i=1}^n a_{n,i} \int_{\mathbb{R}^d}  \Phi \cdot  dD \chi_{\Omega_{n,i}}
\end{align}
for every $\Phi \in C_0(\mathbb{R}^d;\mathbb{R}^d)$ and, moreover,
\begin{align}\label{ell_1}
\limsup_{n\to \infty} \sum_{i=1}^n |a_{n,i}| |D\chi_{\Omega_{n,i}}|(\mathbb{R}^d) \leq \|\nabla u\|_{L^1(\mathbb{R}^d)}.
\end{align}
\end{lemma}

\begin{proof}
We begin by recalling the coarea formula for smooth functions with summable gradient:
\begin{equation}
\label{Coarea}
\int_{\mathbb{R}^d} |\nabla u|\;dx = \int\limits_\mathbb{R}\mathcal{H}^{d-1}\Big(\set{x\in \mathbb{R}^d}{u(x) = t}\Big)\,dt.
\end{equation}
By Sard's theorem, for almost all~$t$, the set~$E_t = \set{x\in \mathbb{R}^d}{u(x) = t}$ does not contain any critical points of~$u$. Using the implicit function theorem and the property that~$u$ is compactly supported, we have that the function~$t\mapsto \mathcal{H}^{d-1}(E_t)$ is continuous at such points. Therefore, the integral on the right hand side of~\eqref{Coarea} is an improper Riemann integral. Thus, there exist increasing sequences~$\{t_{n,i}\}_{n\in \mathbb{N},i=1,\ldots,n}$ such that
\begin{align*}
\max_{i=1,\ldots,n} (t_{n,i+1} - &t_{n,i}) \to 0,\quad n\to \infty, \quad \text{and}\\
&\int\limits_\mathbb{R}\mathcal{H}^{d-1}\Big(\set{x\in \mathbb{R}^d}{u(x) = t}\Big)\,dt = \lim\limits_{n\to \infty}\sum\limits_{i=1}^n (t_{n,i+1} - t_{n,i})\mathcal{H}^{d-1}\big(E_{t_{n,i}}\big).
\end{align*}
Set~$a_{n,i} = t_{n,i+1} - t_{n,i}$,~$\Omega_{n,i} = \set{x\in\mathbb{R}^d}{u(x) > t_{n,i}}$, and note that the limit relation above leads to the desired~$\ell_1$-estimate~\eqref{ell_1} (since the sets~$E_{t_{n,i}}$ are smooth hypersurfaces, their perimeters coincide with their~$(d-1)$-Hausdorff measures). It remains to prove the~weak--star convergence~\eqref{weak}. Let~$\Phi \in C_0(\mathbb{R}^d,\mathbb{R}^d)$. We wish to show
\begin{align*}
\int_{\mathbb{R}^d} \Phi \cdot \nabla u \;dx = \lim_{n\to \infty} \sum_{i=1}^n (t_{n,i+1} - t_{n,i}) \int_{\mathbb{R}^d}  \Phi \cdot  dD \chi_{\Omega_{n,i}}.
\end{align*}
By the already obtained~\eqref{ell_1}, it suffices to prove the latter limit assertion for~$\Phi$ from a dense subset of~$C_0(\mathbb{R}^d,\mathbb{R}^d)$, and  thus we may assume~$\Phi \in C^1_c(\mathbb{R}^d,\mathbb{R}^d)$. In such a case,
\begin{align*}
\sum_{i=1}^n (t_{n,i+1} - &t_{n,i}) \int_{\mathbb{R}^d}  \Phi \cdot  dD \chi_{\Omega_{n,i}}\\ 
&= -\sum\limits_{i=1}^n(t_{n,i+1}-t_{n,i})\int\limits_{\Omega_{n,i}}\Div \Phi \longrightarrow - \int\limits_{\mathbb{R}^d} u \Div \Phi = \int_{\mathbb{R}^d} \Phi \cdot \nabla u \;dx;
\end{align*}
we may justify the convergence since~$\sum_{i=1}^n (t_{n,i+1} - t_{n,i})\chi_{\Omega_{n,i}}$ tends to~$u$ pointwise and has the summable majorant~$|u| + 1$.
\end{proof}

\section{Proofs of Main Results}\label{proofs}
We begin with the 
\begin{proof}[Proof of Lemma \ref{BoxingLemma}]
For each $x \in U$, there exists $R \in \mathcal{D}$ such that $x \in R \subseteq U$ since $U$ is open. Also, since $U$ is bounded, it holds that $|U\cap R|/|R|\rightarrow 0$ whenever $|R|\rightarrow \infty$. Therefore, for each $x \in U$, there exists a minimal $R_x \in \mathcal{D}$ containing $x$ such that $|U \cap R_x|/|R_x| \ge \frac{1}{2}$ and $|U\cap \widehat{R}_x|/|\widehat{R}_x| < \frac{1}{2}$, where $\widehat{R}_x$ is the parent cube of $R_x$. 
Let $\{Q_j\}_{j=1}^{\infty}$ be the collection of maximal cubes in $\{\widehat{R}_x\}_{x \in U}$. 

It is clear that $\{Q_j\}_{j=1}^{\infty}\subseteq\mathcal{D}$ is a disjoint cover of $U$. To verify the last condition, we use the following formulation of the Poincar\'e inequality: there exists $c>0$ such that 
$$
    \fint_Q\fint_Q |\chi_U(y) - \chi_U(z)|\,dydz \leq c \ell(Q)^{1-d}|D\chi_U|(Q)
$$
for every cube $Q \subseteq \mathbb{R}^d$. Given $Q_j$, by its selection condition and continuity, there exists a cube $\widetilde{Q}_j$ such that $\widetilde{Q}_j \subsetneq Q_j$, $\ell(Q_j) \leq 2\ell(\widetilde{Q}_j)$, and $ |U\cap \widetilde{Q}_j| /|\widetilde{Q}_j| = |U^c \cap \widetilde{Q}_j| / |\widetilde{Q}_j| = \frac{1}{2}$. Thus, we have that 
$$
    \frac{1}{2} = 2\frac{|U\cap \widetilde{Q}_j|}{|\widetilde{Q}_j|}\frac{|U^c \cap \widetilde{Q}_j|}{|\widetilde{Q}_j|} = \fint_{\widetilde{Q}_j}\fint_{\widetilde{Q}_j} |\chi_U(y) - \chi_U(z)|\,dydz \leq c \ell(\widetilde{Q}_j)^{1-d}|D\chi_U|(\widetilde{Q}_j),
$$
and, equivalently, 
$$
    \ell(\widetilde{Q}_j)^{d-1}\leq 2c|D\chi_U|(\widetilde{Q}_j).
$$
By the above properties and the monotonicity of the measure $|D\chi_U|$, we have 
$$
    \ell(Q_j)^{d-1}\leq 2^{d-1} \ell(\widetilde{Q}_j)^{d-1} \leq 2^dc|D\chi_U|(\widetilde{Q}_j) \leq 2^dc|D\chi_U|(Q_j).
$$
Hence, the result holds with $C=2^dc$.
\end{proof}

We next give the
\begin{proof}[Proof of Theorem \ref{AtomicDecomposition}]
Let $u \in \dot{\BV}(\mathbb{R}^d)$.   By mollification and multiplication by cutoff functions, we may find $\{u_k\} \subset C^\infty_c(\mathbb{R}^d)$ such that
\begin{align*}
Du_k &\overset{*}{\rightharpoonup} Du\\
|Du_k|(\mathbb{R}^d) &\to |Du|(\mathbb{R}^d),
\end{align*}
see alternatively the approximation of Bonami and Poornima \cite{BP}. For each $u_k$, an application of Lemma \ref{sampling} gives
\begin{align*}
\nabla u_k = \lim_{n \to \infty} \sum_{i=1}^n a_{n,i,k} D \chi_{\Omega_{n,i,k}}
\end{align*}
weakly--star in the sense of measures and
\begin{align*}
\limsup_{n\to \infty} \sum_i |a_{n,i,k}| |D\chi_{\Omega_{n,i,k}}|(\mathbb{R}^d) \leq \|\nabla u_k\|_{L^1(\mathbb{R}^d)}.
\end{align*}
For brevity of notation, write $\Omega= \Omega_{n,i,k}$ and apply Lemma \ref{BoxingLemma} to obtain cubes $Q_j=Q_{n,i,k,j}$ such that
\begin{align*}
\chi_\Omega = \sum_{j=1}^\infty \chi_{\Omega  \cap Q_j}
\end{align*}
    and
\begin{align*}
        \ell(Q_j)^{d-1} \leq C|D\chi_\Omega|(Q_j).
\end{align*}
Putting these several facts together we have
\begin{align*}
D u = \lim_{k \to \infty} \lim_{n \to \infty} \lim_{M\to \infty} \sum_{i=1}^n  \sum_{j=1}^M a_{n,i,k} D\chi_{\Omega_{n,i,k}  \cap Q_{n,i,k,j}}
\end{align*}
weakly--star in the sense of measures.

Define
$$
	\lambda_{n,i,k,j} = C'a_{n,i,k} |D\chi_{\Omega_{n,i,k}}|(Q_{n,i,k,j})
$$
and
$$
	\mu^l_{n,i,k,j} = \frac{1}{C'}\frac{ [D\chi_{\Omega_{n,i,k}\cap Q_{n,i,k,j}}]_l}{|D\chi_{\Omega_{n,i,k}}|(Q_{n,i,k,j})}
$$
for $C'>0$ to be chosen.  We claim $ \mu^l_{n,i,k,j}$ are $(d-1)$-atoms, i.e.
\begin{enumerate}
\item $\operatorname*{supp} \mu^l_{n,i,k,j} \subset 2Q_{n,i,k,j}$;
\item $\mu^l_{n,i,k,j}(2Q_{n,i,k,j}) = 0$;
\item $\esssup_{x \in \mathbb{R}^d, t>0} \left|p_t\ast \mu^l_{n,i,k,j} \right| \leq \frac{1}{t^{1/2}}\frac{1}{l(2Q_{n,i,k,j})^{d-1}}$; and
\item $|\mu^l_{n,i,k,j}|(\mathbb{R}^d) \leq 1$.
\end{enumerate}
The first two properties are straightforward, while the third and fourth follow from minor computations and the inequality $\ell(Q_{n,i,k,j})^{d-1} \leq C|D\chi_{\Omega_{n,i,k}}|(Q_{n,i,k,j})$.  In particular, concerning the third property one has
\begin{align*}
   \left|p_t\ast \mu^l_{n,i,k,j}  \right| &= \frac{1}{C'} \left|p_t\ast \frac{ [D \chi_{\Omega_{n,i,k}\cap Q_{n,i,k,j}}]_l}{|D\chi_{\Omega_{n,i,k}}|(Q_{n,i,k,j})} \right|  \\
   &\leq \frac{C2^{d-1}}{C'}\left|\frac{\partial p_t}{\partial x_l} \ast \frac{ \chi_{\Omega_{n,i,k}\cap Q_{n,i,k,j}}}{l(2Q_{n,i,k,j})^{d-1}} \right| \\
   &\leq  \frac{C2^{d-1}t^{1/2}\|\nabla p_t\|_{L^1}}{C'} \frac{1}{t^{1/2}} \frac{1}{l(2Q_{n,i,k,j})^{d-1}},
\end{align*}
so that one should choose $C'$ large enough such that $C2^{d-1}\|\nabla p_1\|_{L^1}\leq C'$.  The fourth property can be argued similarly, as the product rule in $\dot{\BV}\cap L^\infty$ implies
\begin{align*}
|\mu^l_{n,i,k,j}|(\mathbb{R}^d) &= |\mu^l_{n,i,k,j}|(2Q_{n,i,k,j}) \\
&\leq \frac{1}{C'}\frac{ |D\chi_{\Omega_{n,i,k}}|(Q_{n,i,k,j}) + l(Q_{n,i,k,j})^{d-1}}{|D\chi_{\Omega_{n,i,k}}|(Q_{n,i,k,j})} \\
&\leq \frac{1+C}{C'}\frac{ |D\chi_{\Omega_{n,i,k}}|(Q_{n,i,k,j}) }{|D\chi_{\Omega_{n,i,k}}|(Q_{n,i,k,j})},\end{align*}
so that we should also choose $C'>1+C$.  One then further observes that
\begin{align*}
\limsup_{k \to \infty} \limsup_{n \to \infty} \limsup_{M\to \infty}\sum_{i=1}^n  \sum_{j=1}^M |\lambda_{n,i,k,j}| &\leq  \limsup_{k \to \infty} \limsup_{n \to \infty} \sum_{i=1}^n C'a_{n,i,k} |D\chi_{\Omega_{n,i,k}}|(\mathbb{R}^d)\\
&\leq C' \limsup_{k \to \infty}  \|\nabla u_k\|_{L^1(\mathbb{R}^d)}\\
&= C' |Du|(\mathbb{R}^d).
\end{align*}
As $C_0(\mathbb{R}^d)$ is a separable space, the unit ball of measures in the weak--star topology is metrizable.  A diagonalization argument then gives the desired decomposition.  This completes the proof.
\end{proof}

Corollary \ref{dimension_stable_embedding} is an immediate consequence of Theorem \ref{AtomicDecomposition} and the definition of the space $DS_{d-1}(\mathbb{R}^d)$.  As a result, Corollary \ref{ac} follows from \cite[Theorem G]{DS}.  It therefore remains to give the 
\begin{proof}[Proof of Corollaries \ref{AtomicDecompositionBV}, \ref{GN}, \ref{ac_omega}, and \ref{BV_trace}]
Let $u \in \BV(\Omega)$.  By \cite[Proposition 3.21 on p.~131]{AFP}, there exists $\overline{u}$ such that $\overline{u}=u$ in $\Omega$ and
\begin{align}\label{extension_inequality}
|D\overline{u}|(\mathbb{R}^d) \lesssim |Du|(\Omega) + \|u\|_{L^1(\Omega)}.
\end{align}
This inequality and an application of Theorem \ref{AtomicDecomposition} to $\overline{u} \in \dot{\BV}(\mathbb{R}^d)$ yields the conclusion of Corollary \ref{AtomicDecompositionBV}.  

The embedding properties of Lebesgue spaces and \cite[Theorem A]{DS} give the inequality
\begin{align*}
    \|I_1 [D\overline{u}]_l\|_{L^{d/(d-1)}(\mathbb{R}^d;\mathbb{R}^d)} \lesssim \|[D\overline{u}]_l\|_{DS_{d-1}(\mathbb{R}^d)},
\end{align*}
which in combination with Corollary \ref{ac}, the boundedness of the Riesz transforms, and the extension inequality \eqref{extension_inequality} yields
\begin{align*}
\|\overline{u}\|_{L^{d/(d-1)}(\mathbb{R}^d)} \lesssim |Du|(\Omega) + \|u\|_{L^1(\Omega)},
\end{align*}
from which the claim of Corollary \ref{GN} follows from the fact that $\overline{u}=u$ in $\Omega$.

Corollary \ref{ac_omega} follows from the same absolute continuity property established in Corollary \ref{ac} and the fact that $D\overline{u}=Du$ in $\Omega$.

Finally, Corollary \ref{BV_trace} follows from \cite[Theorem B]{DS} and the extension inequality \eqref{extension_inequality}.

\end{proof}

\section*{Acknowledgements}
D. Spector is supported by the National Science and Technology Council of Taiwan under research grant number 113-2115-M-003-017-MY3 and the Taiwan Ministry of Education under the Yushan Fellow Program.  D. Stolyarov is supported by the Russian Science Foundation grant n. 19-71-30002.

\end{document}